\documentclass[a4paper,10pt,leqno]{amsart}
\usepackage{amsmath,amsfonts,amssymb,amsthm,epsfig,epstopdf,url,array}
\usepackage{amsthm}
\usepackage{mathrsfs}
\usepackage{epsfig}
\usepackage{graphicx}

\newtheorem{thm}{Theorem}[section]
\newtheorem{prop}[thm]{Proposition}

\newtheorem{rmk}[thm]{Remark}

\numberwithin{equation}{section}

\newcommand{\ord}{\text{ord}}
\newcommand{\N}{{\mathbb{N}}}
\newcommand{\z}{{\mathbb{Z}}}

\newcommand{\eq}{\equiv}
\newcommand{\rom}[1]{\uppercase\expandafter{\romannumeral #1\relax}}

\title[Non-almost regular quaternary $m$-gonal forms]{Non-almost regular quaternary $m$-gonal forms}

\author{Dayoon Park}
\address{Department of Mathematical Sciences, Ulsan National Institute of Science and
Technology, Ulsan, Korea}
\email{pdy1016@unist.ac.kr}
\thanks {This work was supported by the National Research Foundation of Korea (NRF) grant funded by the Korea government(MSIT) (No. 2020R1A4A1016649). \\
This work was supported by the UBSI Fellowship Program (project No.1.210131.01) of UNIST}
\begin{document}

\maketitle

\begin{abstract}
The maximal rank $n$ which admits a non-almost regular $m$-gonal form is $4$.
In this article, for any given $(a_1,a_2,a_3,a_4) \in \N^4$ and sufficiently large $m$, we completely determine whether the $m$-gonal form $\left<a_1,a_2,a_3,a_4\right>_m$ is almost regular or not.
And we show that for any $m \ge 3$, there are infinitely many non-almost regular $m$-gonal forms.
\end{abstract}


\section{Introduction}

The {\it $m$-gonal number} which is defined as a total number of dots to constitute a regular $m$-gon is one of the classical objects in the number theory.
More precisely, we call the total number of dots of regular $m$-gon with $x$ dots for each side
\begin{equation} \label{m number}
    P_m(x):=\frac{m-2}{2}(x^2-x)+x
\end{equation}
as {\it $x$-th $m$-gonal number}.
One of the most famous stories about $m$-gonal number is a Fermat's Conjecture which states that every positive integer may be written as a sum of at most $m$ $m$-gonal numbers.
This conjecture was resolved by Lagrange, Gauss, and Cauchy for $m=4$, $m=3$, and $m\ge 3$ , respectively.
As a generalization of Gauss's work, Liouville classified the $3$-tuples $(a_1,a_2,a_3) \in \N^3$ for which
every positive integer can be written as the weighted sum of triangular numbers $a_1P_3(x_1)+a_2P_3(x_2)+a_3P_3(x_3)$.
And as a generalization of Lagrange's work, Ramanujan classified the $4$-tuples $(a_1,a_2,a_3,a_4) \in \N^4$ for which
every positive integer can be written as the weighted sum of square numbers $a_1P_4(x_1)+a_2P_4(x_2)+a_3P_4(x_3)+a_4P_4(x_4)$, but one mistake was found later in the Ramanujan's list.

We call a weighted sum of $m$-gonal numbers
\begin{equation} \label{m form}
    F_m(\mathbf x)=a_1P_m(x_1)+\cdots+a_nP_m(x_n)
\end{equation}
where $(a_1,\cdots,a_n) \in \N^n$ as {\it $m$-gonal form}.
We simply write the $m$-gonal form $a_1P_m(x_1)+\cdots+a_nP_m(x_n)$ as $\left<a_1,\cdots,a_n\right>_m$.
But exceptionally we adopt a notation $\left<a_1,\cdots,a_n\right>$ for the square (or $4$-gonal) form (i.e., diagonal quadratic form) instead of $\left<a_1,\cdots,a_n\right>_4$ true to the tradition of the theory of quadratic form.

By directly following the definition of original $m$-gonal number, we may catch that only non-negative rational integer $x$ would be admitted in \eqref{m number}.
But recently, as a generalized $m$-gonal number version, we often admit negative rational integer $x$ too.
For $N \in \N$, if the diophantine equation 
\begin{equation} \label{rep}
F_m(\mathbf x)=N    
\end{equation}
has an integer solution $\mathbf x \in \z^n$ (in sense of the generalized $m$-gonal number) (resp. a non-negative integer solution $\mathbf x \in \N_0^n$ (in sense of the original $m$-gonal number)), then we say that  the $m$-gonal form $F_m(\mathbf x)$ {\it (globally) represents} $N$ over $\z$ (resp. over $\N_0$).
Undoubtedly, a representation of an integer $N$ over $\N_0$ implies a representation of $N$ over $\z$ by an $m$-gonal form.
As is well known, the problem of completely classifying every positive integer $N$ which is represented by arbitrary given $m$-gonal form $F_m(\mathbf x)$ is not easy in general.
In an effort to approach the problem,
we suggest the reduced congruence equation 
\begin{equation} \label{loc rep1}
    F_m(\mathbf x)\eq N \pmod{r}.
\end{equation}
When the congruence equation \eqref{loc rep1} has an integer solution $\mathbf x \in \z^n$ (or equivalently, has a non-negative integer solution $\mathbf x \in \N_0^n$) for every $r \in \z$, we say that $F_m(\mathbf x)$ {\it locally represents} $N$.
The solvability of congruence equation 
$$F_m(\mathbf x)\eq N \pmod {p^{\alpha}}$$
only for every prime $p$ and natural number ${\alpha} \in \N$
would be enough to derive the local representability of $N$ by $F_m(\mathbf x)$ in virtue of the Chinese Remainder Theorem.
Note that for a prime $p$,
the congruence equation $F_m(\mathbf x)\eq N \pmod {p^{\alpha}}$
has an integer solution $\mathbf x \in \z^n$ for every natural number ${\alpha} \in \N$
if and only if the equation $F_m(\mathbf x)=N$ has a $p$-adic integer solution $\mathbf x \in \z_p^n$.
When the equation $F_m(\mathbf x)=N$ has a $p$-adic integer solution $\mathbf x \in \z_p^n$, we say that {\it $F_m(\mathbf x)$ represents $N$ over $\z_p$}.
Obviously, the (global) representability implies the local representability, but the converse does not hold in general.
Especially, when the converse also holds, i.e., an $m$-gonal form $F_m(\mathbf x)$ represents every positive integer $N$ which is locally represented by $F_m(\mathbf x)$ over $\z$ (resp. over $\N_0$), we say that $F_m(\mathbf x)$ is {\it regular over $\z$ (resp. over $\N_0$)}.
When $F_m(\mathbf x)$ represents every positive integer $N$ which is locally represented by $F_m(\mathbf x)$ but finitely many over $\z$ (resp. over $\N_0$), we say that $F_m(\mathbf x)$ is {\it almost regular over $\z$ (resp. over $\N_0$)}.

The concept of {\it regular form} was firstly introduced by Dickson in \cite{D}.
In \cite{W}, Watson proved that there are only finitely many primitive positive definite regular ternary quadratic forms by introducing very nice {\it Watson's $\Lambda$- transformation} which has been widely applied in the study of regular form.
In \cite{JKS}, Jagy, Kaplansky, and Schiemann who used the Watson's $\Lambda$-transformation classified 913 candidates for primitive positive definite regular ternary quadratic forms.
And they claimed the regularity of all of them but 22.
Oh \cite{Oh} showed the regularity of $8$ forms of unconfirmed $22$ forms.
And Lemke Oliver \cite{LO} proved the regularity of the remaining $14$ candidates under the assumption of generalized Riemann
Hypothesis.

The finiteness of primitive regular ternary triangular forms (i.e., $3$-gonal forms) was shown by Chan and Oh \cite{CO} and all the finitely many primitive regular ternary triangular forms was determined by Kim and Oh \cite{KO}.
As an improvement of the result in \cite{CO}, Chan and Ricci \cite{CR} proved the finiteness of primitive regular ternary quadratic polynomials with a fixed conductor in \cite{CR}.
As another version's improvement of the work, in \cite{CO}, He and Kane showed the finiteness of primitive regular ternary polygonal forms (i.e., $m$-gonal forms for arbitrary $m \ge 3$).
Which implies that there are only finitely many $m$ which admits a regular $m$-gonal form of rank $3$.
In \cite{reg}, Kim and the author showed that
for each rank $n \ge 4$ too, there are only finitely many $m$ which admits a regular $m$-gonal form of rank $n$.
Moreover in \cite{reg}, they completely determined the type of regular $m$-gonal forms of rank $n \ge 4$ for all $m \ge 14$ with $m \not\eq 0 \pmod 4$ and $m \ge 28$ with $m \eq 0 \pmod 4$
by introducing {\it $\Lambda_p$-transformation} which is well designed to track the regular $m$-gonal forms.

\vskip 0.8em

Now let us turn to our main subject which is almost regular form in this article.
By Theorem 4.9 (1) in \cite{CO}, every $m$-gonal form of rank $n \ge 5$ is almost regular over $\z$, i.e., 
every $m$-gonal form of rank $n \ge 5$ represents all the sufficiently large integers (in other word, all but finitely many integers) which are locally represented by the form over $\z$.
But such the nice local-to-global principle is not guaranteed when the rank is less than 5 in general.
There is non-almost regular $m$-gonal form of rank $4$, and by the above argument, $4$ should be the maximum rank which admits a non-almost regular $m$-gonal form over $\z$.
In this short article, we mainly look into the non-almost regular $m$-gonal forms of rank $4$.
The following two theorems are the main goal in this article.
\begin{thm} \label{thm 1}
Let $a_1,a_2,a_3,a_4 \in \N$ be given.
\begin{itemize}
    \item [(1) ] If the quaternary quadratic form $\left<a_1,a_2,a_3,a_4\right>\otimes \z_p$ is isotropic for every prime $p$, then 
    the quaternary $m$-gonal form $\left<a_1,a_2,a_3,a_4\right>_m$ is almost regular over $\z$ for any $m \ge 3$.
    \item[(2) ] Otherwise, define $T:=\{p| \text{$\left<a_1,a_2,a_3,a_4\right>\otimes \z_p$ is anisotropic}\}$ a non-empty set of primes.\\
For any $m \ge 3$ satisfying
\begin{equation} \label{cond}
\begin{cases}
p|m-2 & \text{for all $p \in T \setminus\{2\}$}\\
m \not\eq 0 \pmod{4} & \text{if } 2 \in T,
\end{cases}
\end{equation}
the $m$-gonal form $\left<a_1,a_2,a_3,a_4\right>_m$ is almost regular over $\z$.\\
For every sufficiently large $m$ which is not satisfying \eqref{cond}, the $m$-gonal form $\left<a_1,a_2,a_3,a_4\right>_m$ is not almost regular over $\z$ (so also over $\N_0$).
\end{itemize}
\end{thm}

\vskip 0.8em

\begin{thm}\label{thm 2}
For any $m \ge 3$, there are infinitely many primitive quaternary $m$-gonal forms which are not almost regular over $\z$ (so also over $\N_0$).
\end{thm}

\vskip 0.8em

\begin{rmk}
By using Theorem \ref{thm 1}, for arbitrary given $4$-tuple $(a_1,a_2,a_3,a_4) \in \N^4$ and any sufficiently large $m$, we may directly determine whether the $m$-gonal form $\left<a_1,a_2,a_3,a_4\right>_m$ is almost regular or not.
\end{rmk}

\vskip 0.8em

Throughout this article, we adopt the arithmetic theory of quadratic forms.
Before we move on, we briefly set our languages and terminologies which are used throughout this article.

If a quadratic form (or $m$-gonal form) represents every $p$-adic integer in $N\z_p$ over $\z_p$, then we say that the quadratic form (or $m$-gonal form) is {\it $N\z_p$-universal} and especially, when $N\z_p=\z_p$, we say that {\it universal over $\z_p$}.
If a quadratic form (or $m$-gonal form) is universal over $\z_p$ for every prime $p$, then we say that the quadratic form (or $m$-gonal form) is {\it locally universal}.

We say that a quadratic form $Q(\mathbf x)$ is {\it isotropic} if the equation $Q(\mathbf x)=0$ has a non-trivial solution $\mathbf x \not\eq \mathbf 0$ and $\it anisotropic$ otherwise.
It is well known that any quadratic form $Q(\mathbf x)$ of rank $n\ge 5$ is isotropic over local ring $\z_p$.
This property is critical in the regularity of arbitrary quadratic form of rank $n \ge 5$.

Any unexplained notation and terminology can be found in \cite{O}.


\section{Preliminaries}
First of all, we may get an easy but important observation that
$$N=a_1P_m(x_1)+\cdots+a_nP_m(x_n)$$
if and only if
$$8(m-2)N+(m-4)^2(a_1+\cdots+a_n)=\sum_{i=1}^na_i(2(m-2)x_i-(m-4))^2.$$
Which allows transfer the issue of the representation by $m$-gonal form to the issue of the representation by diagonal quadratic form with a congruence condition.
By using the following proposition, one may effectively determine the representability of integer by an $m$-gonal form over non-archimedean local ring with help of known results in the quadratic form theory over non-archimedean local ring.

\begin{prop} \label{loc rep}
Let $F_m(\mathbf x)=a_1P_m(x_1)+\cdots+a_nP_m(x_n)$ be a primitive (i.e., satisfying $(a_1,\cdots,a_n)=1$) $m$-gonal form.
\begin{itemize}
    \item [(1) ] When $p$ is an odd prime with $p|m-2$, $F_m(\mathbf x)$ is universal over $\z_p$.
    \item [(2) ] When $m \not\eq 0 \pmod 4$, $F_m(\mathbf x)$ is universal over $\z_2$.
   \item [(3) ] When $p$ is an odd prime with $(p,m-2)=1$, an integer $N$ is represented by $F_m(\mathbf x)$ over $\z_p$ if and only if the integer $8(m-2)N+(a_1+\cdots+a_n)(m-4)^2$ is represented by the diagonal quadratic form $\left<a_1,\cdots,a_n \right>$ over $\z_p$.
    \item [(4) ] When $m \eq 0 \pmod 4$, an integer $N$ is represented by $F_m(\mathbf x)$ over $\z_2$ if and only if the integer $\frac{m-2}{2}N+(a_1+\cdots+a_n)\left(\frac{m-4}{4}\right)^2$ is represented by the diagonal quadratic form $\left<a_1,\cdots,a_n \right>$ over $\z_2$.
\end{itemize}
\end{prop}
\begin{proof}
See Proposition 3.1 in \cite{rank 5}.
\end{proof}

\vskip 0.8em

\begin{rmk} \label{aniso}
Now, we may need understanding the behavior of representation by a diagonal quadratic form which is anisotropic over a non-archimedean local ring.

\begin{itemize}
\item[(1) ] First, we consider a quaternary quadratic form which is anisotropic over a non-dyadic (i.e., for an odd prime $p$) local ring $\z_p$.
A quaternary anisotropic diagonal quadratic form over $\z_p$ is the form of
$$\left<ap^{2r_1+1}, aup^{2r_2+1}, bp^{2r_3}, bu'p^{2r_4}\right>$$
where $a,b \in \z_p^{\times}$, $r_i \in \N_0$, and $-u$ and $-u'$ are quadratic non-residues in $\z_p^{\times}$. 
One may see that for a quadratic form $\left<a_1,a_2,a_3,a_4\right>$ which is anisotropic over $\z_p$, when
$$\ord_p(a_1x_1^2+a_2x_2^2+a_3x_3^2+a_4x_4^2)\ge r$$
for some $(x_1,x_2,x_3,x_4)\in\z_p^4$ (so also $ \in \z^4$ or $ \in \N_0^4$),
$$\ord_p(a_ix_i^2)\ge r$$ for each $i=1,2,3,4$.

\item[(2) ]A primitive quaternary anisotropic diagonal quadratic form over dyadic $\z_2$ is the form of either
$$\left<a_12^{2r_1},a_22^{2r_2},a_32^{2r_3+1},a_42^{2r_4+1}\right>$$ 
where $r_1,r_2,r_3,r_4 \in \N_0$ and $a_1,a_2,a_3,a_4 \in \z_2^{\times}$ with $a_1\eq a_2 \eq \frac{a_3+a_4}{2} \pmod 4$ and $a_3\eq a_4\eq \frac{a_1+a_2}{2} \pmod 4$ or
$$\left<a_12^{2r_1},a_22^{2r_2},a_32^{2r_3},a_42^{2r_4}\right>$$ 
where $r_1,r_2,r_3,r_4 \in \N_0$ and $a_1,a_2,a_3,a_4 \in \z_2^{\times}$ with $a_1\eq a_2 \eq a_3 \eq a_4 \pmod 4$ and $a_1+a_2+a_3+a_4 \not\eq 0 \pmod 8$.
One may get that for a quadratic form $\left<a_1,a_2,a_3,a_4\right>$ which is anisotropic over $\z_2$, when
$$\ord_2(a_1x_1^2+a_2x_2^2+a_3x_3^2+a_4x_4^2)\ge r+2$$
for some $(x_1,x_2,x_3,x_4)\in\z_2^4$ (so also $ \in \z^4$ or $ \in \N_0^4$),
$$\ord_2(a_ix_i^2)\ge r$$ for each $i=1,2,3,4$.

\item[(3) ] The maximal rank $n$ which admits an anisotropic quadratic form over non-archimedean local ring is $4$.
And a diagonal anisotropic quadratic form over a non-archimedean local ring $\z_p$ of rank $n$ (in effect, $n$ should be $1,2,3$, or $4$) would be a subform of a diagonal anisotropic quadratic form of rank $4$.
So from the above (1) and (2), one may obtain that for a diagonal quadratic form
$$\left<a_1,\cdots,a_n\right>$$
which is anisotropic over a non-archimedean local ring $\z_p$, if
$$\ord_p(a_1x_1^2+\cdots+a_nx_n^2)\ge r+(1+(-1)^p),$$
then $$\ord_p(a_ix_i^2)\ge r$$ for all $1 \le i \le n$.
\end{itemize}
\end{rmk}

\vskip 0.8em

\begin{prop} \label{max order}
A diagonal quaternary quadratic form $$\left<a_1,a_2,a_3,a_4\right>$$ is $p^{\max \limits_{1\le i \le 4}\ord_p(a_i)}\z_p$-universal over $\z_p$ for every prime $p$.
\end{prop}
\begin{proof}
First, we assume that $p$ is an odd prime.
By 92:1b in \cite{O}, a ternary unimodular (i.e., $u_1,u_2,u_3 \in \z_p^{\times}$) quadratic form $\left<u_1,u_2,u_3\right>\otimes \z_p$ is $\z_p$-universal and a binary unimodular (i.e., $u_1,u_2 \in \z_p^{\times}$) quadratic form $\left<u_1,u_2\right>\otimes \z_p$ represents every $p$-adic unit $\z_p^{\times}$.
Which may yield the claim.

Next we assume that $p=2$.
One may show that for arbitrary $2$-adic unit $u_i \in \z_2^{\times}$,
the quadratic forms 
$$\left<u_1,2u_2,2u_3,2u_4\right>, \ \left<u_1,u_2,2u_3,2u_4\right>, \ \left<u_1,u_2,u_3,2u_4\right>, \ \left<u_1,u_2,u_3,u_4\right>$$
are $\z_2$-universal over $\z_2$.
For two $2$-adic units $u,u' \in \z_2^{\times}$, since 
$$u\eq u' \pmod{8} \iff u\in u' (\z_2^{\times})^2,$$
it would be enough to show the claim only for $u_i \in \{1,3,5,7\}$.
The tedious processing is omitted in this article.
\end{proof}


\section{Non-almost regular quaternary $m$-gonal forms}

Now we are ready to prove our main goals.
At the end, in this last section, we prove Theorem \ref{thm 1} and Theorem \ref{thm 2} by concretely constructing a non-almost regular $m$-gonal form of rank $4$. 

\vskip 0.8em

\begin{proof} [Proof of Theorem \ref{thm 1}]
Since the $m$-gonal form $\left<a_1,a_2,a_3,a_4\right>_m$ is almost regular if and only if the $m$-gonal form $\left<\alpha a_1,\alpha a_2,\alpha a_3,\alpha a_4\right>_m$ is almost regular for $\alpha \in \N$, throughout this proof, we assume that $(a_1,a_2,a_3,a_4)=1$.\\
(1) Theorem 4.9 (2) in \cite{CO} directly yields the claim.\\
(2) Since the representation over $\N_0$ implies the representation over $\z$, it would be enough to show this theorem only over $\z$.

For an odd prime $p$, if a primitive (i.e., satisfying $(a_1,a_2,a_3,a_4)=1$) quadratic form $\left<a_1,a_2,a_3,a_4\right>$ is anisotropic over $\z_p$, then we have that $a_1+a_2+a_3+a_4 \in \z_p^{\times}$.
When a primitive quadratic form $\left<a_1,a_2,a_3,a_4\right>$ is anisotropic over $\z_2$, we have that 
$a_1+a_2+a_3+a_4 \in 4\z_2^{\times}\cup 2\z_2^{\times}\cup \z_2^{\times}$.
Therefore, we may see that for $m$ satisfying \eqref{cond}, $$\ord_p(8(m-2)N+(m-4)^2(a_1+a_2+a_3+a_4)) \le \ord_p(16) \le 4$$
holds for any $p \in T$ and $N \in \N$, that is to say, the order of integer $8(m-2)N+(m-4)^2(a_1+a_2+a_3+a_4)$ is absolutely bounded at every prime spot $p$ such that $\left<a_1,a_2,a_3,a_4\right>\otimes \z_p$ is anisotropic.
And so Theorem 4.9 (2) in \cite{CO} may yield the first statement.

For the second statement, define $T':=\{p|\left<a_1,a_2,a_3,a_4\right>\otimes \z_p \text{ is not universal}\}$ ($T'$ could be empty).
We show that for any sufficiently large $m$ such that $$m \ge \left(\prod \limits_{p \in T}p^{\max \limits _{1 \le i \le 4}\ord_p(64a_i)}\right)\cdot \left(\prod \limits_{p \in T'\setminus T}p^{\ord_p(4p)}\right)+a_1+a_2+a_3+a_4+4,$$
if $m$ does not satisfy \eqref{cond}, then the $m$-gonal form $\left<a_1,a_2,a_3,a_4\right>_m$ is not almost regular.
If there is an odd prime in $T$ which does not divide $m-2$, then we write the prime as $q \in T$ and
otherwise, then we put $q=2 \in T$ (note that when $q=2$, we have $m \eq 0 \pmod 4$).
We may take a residue $r(q)$ in $\z/q^{\max \limits _{1 \le i \le 4}\ord_q(64a_i)}\z$
for which
$$8(m-2)\cdot r(q)+(m-4)^2(a_1+a_2+a_3+a_4) =0$$  in $\z/q^{\max \limits _{1 \le i \le 4}\ord_q(64a_i)}\z$.
For a prime $p \not\in T'$, the $m$-gonal form $\left<a_1,a_2,a_3,a_4\right>_m$ is universal over $\z_p$ by Proposition \ref{loc rep}.
For a prime $p \in T'$, from our assumption that $(a_1,a_2,a_3,a_4)=1$, we may take a residue $r(p)$ in $\z/p^{\ord_p(4p)}\z$ for which
$N$ is represented by $\left<a_1,a_2,a_3,a_4\right>_m$ over $\z_p$ for any $N \eq r(p) \pmod {p^{\ord_p(4p)}}$.
By combining the above arguments with the Chinese Remainder Theorem, we may take a residue $r$ in $\z/\left(\prod \limits_{p \in T\setminus \{q\}}p^{\ord_p(4p)}\right) \cdot \left(q^{\max \limits _{1 \le i \le 4}\ord_q(64a_i)}\right) \z$
satisfying that for any integer $N$ with $$N \eq r \pmod{\left(\prod \limits_{p \in T'\setminus \{q\}}p^{\ord_p(4p)}\right) \cdot \left(q^{\max \limits _{1 \le i \le 4}\ord_q(64a_i)}\right)},$$
the followings
$$\begin{cases}
\text{$N$ is locally represented by $\left<a_1,a_2,a_3,a_4\right>_m$} \\
\text{$8(m-2)N+(m-4)^2(a_1+a_2+a_3+a_4) \eq 0 \pmod{q^{\max \limits _{1 \le i \le 4}\ord_q(64a_i)}}$}
\end{cases}$$
hold.
So we may take a positive integer $N_0$ in the interval $[a_1+a_2+a_3+a_4+1, m-4]$ such that $N_0$ is locally represented by $\left<a_1,a_2,a_3,a_4\right>_m$ and
$$8(m-2)N_0+(m-4)^2(a_1+a_2+a_3+a_4) \eq 0 \pmod{q^{\max \limits _{1 \le i \le 4}\ord_q(64a_i)}}.$$
But since the smallest $m$-gonal number is $m-3$ except $0$ and $1$, the integer $N_0$ in the interval $[a_1+a_2+a_3+a_4+1, m-4]$ would not be (globally)
 represented by $\left<a_1,a_2,a_3,a_4\right>_m$ over $\z$ (and so over $\N_0$ too).
And then we may get infinitely many positive integers $N_n$ which are represented by $\left<a_1,a_2,a_3,a_4\right>_m$ locally, but not globally 
where
$8(m-2)N_n+(m-4)^2(a_1+a_2+a_3+a_4)=q^{2kn}(8(m-2)N_0+(m-4)^2(a_1+a_2+a_3+a_4))$ for $k \in \N$ satisfying
\begin{equation} \label{k cond}
\begin{cases}
q^k \eq 1 \pmod {8(m-2)}  & \text{when }q \text{ is an odd prime} \\
 q^k \eq 1 \pmod{\frac{m-2}{2}} & \text{when }q=2.
\end{cases}
\end{equation}
For a contradiction, we suppose that $N_n$ is represented by $\left<a_1,a_2,a_3,a_4\right>_m$ over $\z$, i.e.,
$$8(m-2)N_n+(m-4)^2(a_1+a_2+a_3+a_4)=\sum_{i=1}^4a_i(2(m-2)x_i-(m-4))^2$$
for some $(x_1,x_2,x_3,x_4) \in \z^4$.
Since 
$$8(m-2)N_n+(m-4)^2(a_1+a_2+a_3+a_4)\eq 0 \pmod{q^{2kn+{\max \limits_{1\le i \le 4}\{\ord _q(64a_i)\}}}}$$
and $\left<a_1,a_2,a_3,a_4\right>\otimes \z_q$ is anisotropic,
we have that
$$2(m-2)x_i-(m-4) \eq 0 \pmod {q^{kn+\ord_q(4)}}$$
by Remark \ref{aniso}.
From our choice of $k \in \N$ satisfying \eqref{k cond}, we have that
$$\frac{2(m-2)x_i-(m-4)}{q^{kn}}=2(m-2)y_i-(m-4)$$ for some $y_i \in \z$.
And then we obtain a contradiction that $N_0=a_1P_m(y_1)+a_2P_m(y_2)+a_3P_m(y_3)+a_4P_m(y_4)$.
This completes the proof.

\end{proof}

\vskip 0.8em

\begin{rmk}
(1) For $a_1,\cdots,a_n \in \N$ with $(a_1,\cdots,a_n)=1$ and a prime $q$, let $\left<a_1,\cdots,a_n\right>\otimes \z_q$ be anisotropic.

When an $m$-gonal form $\left<a_1,\cdots,a_n\right>_m$ represents a positive integer $N$, 
i.e., 
$$8(m-2)N+(m-4)^2(a_1+\cdots+a_n)=\sum_{i=1}^na_i(2(m-2)x_i-(m-4))^2$$
for some $(x_1,\cdots,x_n) \in \z^4$
for which $8(m-2)N+(m-4)^2(a_1+\cdots+a_n)$ is divided by a high power of $q$, we have that $2(m-2)x_i-(m-4)$ also would be divided by a high power of $q$ since
$$\ord_q(8(m-2)N+(m-4)^2(a_1+\cdots+a_n)) \le \ord_q(4a_i(2(m-2)x_i-(m-4))^2)$$
for each $i=1,\cdots,n$ by Remark \ref{aniso}.
Note that the above situation occurs only when $$\text{$(q,2(m-2))=1 \quad$  or $\quad m \eq 0 \pmod 4$ with $q=2$}$$ because $\ord_q(a_1+\cdots+a_n) \le \ord_q(4)$ for anisotropic $\left<a_1,\cdots,a_n \right>\otimes \z_q$.
So we may get $k \in \N$ for which 
\begin{equation} \label{k cond'}
\begin{cases}
q^k \eq 1 \pmod {8(m-2)}  & \text{when }(q,2(m-2))=1 \\
 q^k \eq 1 \pmod{\frac{m-2}{2}} & \text{when }q=2.
\end{cases}
\end{equation}
When $\ord_q(8(m-2)N+(m-4)^2(a_1+\cdots+a_n))\ge \max \limits_{1\le i \le n} \ord_q(64a_i)+2k$, we may see
that
$$8(m-2)N'+(m-4)^2(a_1+\cdots+a_n)=\sum_{i=1}^na_i(2(m-2)x_i'-(m-4))^2$$
where $N' \in \N$ and $x_i' \in \z$ satisfying
$$\begin{cases}
8(m-2)N'+(m-4)^2(a_1+\cdots+a_n)=\frac{8(m-2)N+(m-4)^2(a_1+\cdots+a_n)}{q^{2k}}\\
2(m-2)x_i'-(m-4)=\frac{2(m-2)x_i-(m-4)}{q^k}
\end{cases}$$
yielding that $N'$ is represented by $\left<a_1,\cdots,a_n\right>_m$ and further a representation of $N$ is derived from a representation of $N'$.
Following the same processing as above, we may back track the representations of the positive integers $N>N'>N''>\cdots>N'^{\cdots}$$'$ at least until  
$$\text{$\max \limits_{1\le i \le n} \ord_q(64a_i) \le \ord_q(8(m-2)N'^{\cdots}$$'+(m-4)^2(a_1+\cdots+a_n)) < \max \limits_{1\le i \le n} \ord_q(64a_i)+2k$}$$
holds.

As the contraposition of the above argument, we may see that when a positive integer $N_0 \in \N$ with $\ord_q(8(m-2)N_0+(m-4)^2(a_1+\cdots+a_n))\ge\max \limits_{1\le i \le n} \ord_q(64a_i)$ is not represented by $\left<a_1,\cdots,a_n\right>_m$ (for which the quadratic form $\left<a_1,\cdots,a_n\right>\otimes \z_q$ is anisotropic), the infinite series of positive integers $N_n$ are not represented by the $\left<a_1,\cdots,a_n\right>_m$ where $8(m-2)N_n+(m-4)^2(a_1+\cdots+a_n)=q^{2kn}(8(m-2)N_0+(m-4)^2(a_1+\cdots+a_n))$ for $k \in \N$ satisfying \eqref{k cond'}.

On the other hand, when $m$ is large, a quaternary $m$-gonal form $\left<a_1,a_2,a_3,a_4\right>_m$ would fail to represent a bunch of small positive integers since the smallest $m$-gonal number is $m-3$ except $0$ and $1$.
As a specific observation, we suggest an example that a quaternary $m$-gonal form $\left<a_1,a_2,a_3,a_4\right>_m$ would represent at most $2^4-1$ positive integers before $m-3$.
Therefore there is  a lot of chance to take a positive integer $N_0$ (especially, among the small positive integers) which is  represented by $\left<a_1,a_2,a_3,a_4\right>_m$ locally, but not globally and satisfy $\ord_q(8(m-2)N_0+(m-4)^2(a_1+a_2+a_3+a_4))\ge\max \limits_{1\le i \le 4} \ord_q(64a_i)$.
\\

(2) Very recently, it was shown that any $m$-gonal form of rank $n \ge 5$ is almost regular over $\N_0$ too (not only over $\z$) in \cite{N_0}.
The almost regularity of $m$-gonal form $\left<a_1,\cdots,a_n\right>_m$ of rank $n \ge 5$ is strongly bound to the property that quadratic form $\left<a_1,\cdots,a_n\right>$ of rank $n \ge 5$ is always isotropic over every non-archimedean local ring $\z_p$.
As the same stream, by Theorem \ref{thm 1} (1), when the positive definite quaternary quadratic form $\left<a_1,a_2,a_3,a_4\right>$ is isotropic over every non-archimedean local ring $\z_p$, the $m$-gonal form $\left<a_1,a_2,a_3,a_4\right>_m$ is almost regular for every $m \ge 3$ over $\z$.
Unfortunately, it is unknown whether such a principle in the case of quaternary form also holds over $\N_0$ yet.
The author suggests the interesting open problems that when a positive definite quaternary quadratic form $\left<a_1,a_2,a_3,a_4\right>$ is isotropic, 
to determine that whether the almost regularity of the m-gonal forms $\left<a_1,a_2,a_3,a_4\right>_m$ are still preserved  or not for $m \ge 3$ over $\N_0$ too and to determine that whether the first statement in Theorem \ref{thm 1} (2) holds over $\N_0$ too or not.
\end{rmk}

\vskip 0.8em

\begin{proof} [Proof of Theorem \ref{thm 2}]
Let $p$ be an odd prime with $(p,m-2)=1$.
We may take an odd prime $q_1$ which is a quadratic residue modulo $p$ and an odd prime $q_2(\not=q_1)$ for which $-q_2$ is a quadratic non-residue modulo $p$. 
Note that it would be possible to take infinitely many such odd primes $p, q_1$ and $q_2$.
Now we show that the $m$-gonal form $$\left<q_1,q_2,q_1p,q_2p\right>_m$$
is not almost regular over $\z$ (so also over $\N_0$).
By Proposition \ref{loc rep}, we may see that $\left<q_1,q_2,q_1p,q_2p\right>_m$ is locally universal since the quaternary quadratic form $\left<q_1,q_2,q_1p,q_2p\right>$ is locally universal.
But the $m$-gonal form $\left<q_1,q_2,q_1p,q_2p\right>_m$ does not (globally) represent the infinitely many positive integers
$$N_n=\frac{p^{2kn}\{8(m-2)+(m-4)^2(q_1+q_2+q_1p+q_2p)\}-(m-4)^2(q_1+q_2+q_1p+q_2p)}{8(m-2)}$$
where $k\in \N$ with $p^{k}\eq 1 \pmod{8(m-2)}$.
When $n=0$, we may easily see that $N_0=1$ is not represented by $\left<q_1,q_2,q_1p,q_2p\right>_m$ since $q_1,q_2,q_1p,q_2p>1$.
When $n>0$, for a contradiction, we suppose that $N_n$ is represented by $\left<q_1,q_2,q_1p,q_2p\right>_m$.
On the other words, we suppose that $q_1(2(m-2)x_1-(m-4))^2+q_2(2(m-2)x_2-(m-4))^2+q_1p(2(m-2)x_3-(m-4))^2+q_2p(2(m-2)x_4-(m-4))^2=8(m-2)N_n+(m-4)^2(q_1+q_2+q_1p+q_2p)$ for some $(x_1,x_2,x_3,x_4) \in \z^4$.
Since the quadratic form $\left<q_1,q_2,q_1p,q_2p\right>$ is anisotropic over $\z_p$ and $p^{2kn}|8(m-2)N_n+(m-4)^2(q_1+q_2+q_1p+q_2p)$, we have that 
$$p^{kn}|2(m-2)x_i-(m-4)$$ for $i=1,2,3,4$.
Moreover, since $p^{kn} \eq 1 \pmod{2(m-2)}$, we have that 
$$2(m-2)y_i-(m-4)=\frac{2(m-2)x_i-(m-4)}{p^{kn}}$$
for some $y_i \in \z$.
So we may obtain that 
$q_1(2(m-2)y_1-(m-4))^2+q_2(2(m-2)y_2-(m-4))^2+q_1p(2(m-2)y_3-(m-4))^2+q_2p(2(m-2)y_4-(m-4))^2=8(m-2)\cdot 1+(m-4)^2(q_1+q_2+q_1p+q_2p)$, yielding that $1$ is represented by $m$-gonal form 
$\left<q_1,q_2,q_1p,q_2p\right>_m$.
Which is a contradiction.
\end{proof}

\end{document}